\documentclass[a4paper,10pt,article]{amsart}

\usepackage{amsmath, amscd,amssymb,amsthm,CJKnumb,verbatim,indentfirst,dsfont,mathrsfs,ipa,graphicx,textcomp,enumerate}
\usepackage[all]{xy}
\usepackage{CJK,ipa,amsmath,booktabs,longtable}
\usepackage[numbers]{natbib}
\usepackage[colorlinks=true,
            linkcolor=blue,
            anchorcolor=blue,
            citecolor=blue]{hyperref}
\usepackage{fancyhdr}
\fancypagestyle{mainFancy}{
\fancyhf{}

\fancyhead[OC]{\fzkai{\@author}}
\fancyhead[EC]{\fzkai{\@title}}
\fancyhead[OR,EL]{\thepage}
}

\renewcommand{\parallel}{\hbox{/\kern -2pt/}}
\begin{CJK*}{GBK}{song}
\end{CJK*}
\newtheorem{theorem}{Theorem}[section]
\newtheorem{lemma}[theorem]{Lemma}
\newtheorem{corollary}[theorem]{Corollary}

\newtheorem{conjecture}[theorem]{Conjecture}

\theoremstyle{definition}
\newtheorem{definition}[theorem]{Definition}
\newtheorem{example}[theorem]{Example}

\theoremstyle{remark}
\newtheorem{remark}[theorem]{Remark}

\numberwithin{equation}{section}

\title{Open Cones and $K$-theory for $\ell^p$ Roe Algebras}
\begin{document}

\author{Jianguo Zhang}

\maketitle

\begin{abstract}
    In this paper, we verify the $\ell^p$ coarse Baum-Connes conjecture for open cones and show that the $K$-theory for $\ell^p$ Roe algebras of open cones is independent of $p\in[1,\infty)$. Combined with the result of T. Fukaya and S.-I. Oguni, we give an application to the class of coarsely convex spaces that includes geodesic Gromov hyperbolic spaces, CAT(0)-spaces, certain Artin groups and Helly groups equipped with the word length metric. \\
    \:\:\:  \\ 
\textit{Keywords.} $K$-theory, $\ell^p$ Roe algebras, open cones.\\
\textit{2020 Mathematics Subject Classification.} Primary 19K99, 46H99.
\end{abstract}

\let\thefootnote\relax\footnotetext{
Address: Research Center for Operator Algebras, East China Normal University, Shanghai 200062, China.\par
E-mail: jgzhang@math.ecnu.edu.cn\par     
This work was partially supported by NSFC (No. 11771143, No. 12171156) and by Shanghai Key Laboratory of Pure Mathematics and Mathematical Practice. }

\section{Introduction}
    The coarse Baum-Connes conjecture provides an algorithm to compute the higher indices of generalized elliptic operators on open Riemannian manifolds which lie in the $K$-theory of Roe algebras associated to underlying manifolds \cite{RoeBookCoarseCohomology, HigsonRoeCBC, YuCBC}. The conjecture has some significant applications to topology and geometry, in particular, to the Novikov conjecture \cite{HigsonRoeBook, WillettYuBook}. The coarse Baum-Connes conjecture has been verified for a large class of spaces, such as open cones \cite{HigsonRoeCBC}, metric spaces with finite asymptotic dimension \cite{YuFAD} and metric spaces which can be coarsely embedded into a Hilbert space \cite{YuEmbedding}, as well as disproved for large spheres \cite{YuFAD} and expanders \cite{HigsonLafforgueSkandalis}. \par
    In recent years, the $\ell^p$ coarse Baum-Connes conjecture for $p\in [1,\infty)$ (cf. Conjecture \ref{CBC}) gained attention, one of the motivations is to compute the $K$-theory for $\ell^p$ Roe algebras (cf. Definition \ref{DefRoeAlg}). In \cite{ZZ}, based on the work of G. Yu \cite{YuFAD}, the author and D. Zhou proved that for $p\in[1,\infty)$, the $\ell^p$ coarse Baum-Connes conjecture holds for spaces with finite asymptotic dimension and the $K$-theory for $\ell^p$ Roe algebras of such spaces is independent of $p\in (1,\infty)$. In \cite{ShanWang}, for $p\in(1,\infty)$, L. Shan and Q. Wang proved that the injective part of the $\ell^p$ coarse Baum-Connes conjecture is true for metric spaces with bounded geometry which can be coarsely embedded into a simply connected complete Riemannian manifold of non-positive sectional curvature. On the other hand, in \cite{ChungNowakPCBC}, Y. C. Chung and P. W. Nowak proved that expanders are still counterexamples to the $\ell^p$ coarse Baum-Connes conjecture for $p\in(1,\infty)$.  \par
    There are also some results in the literature concerning the structure, rigidity and $K$-theory of $\ell^p$ uniform Roe algebras, $L^p$ group algebras and $L^p$ operator crossed products, refer to \cite{ChungDynamical, ChungLiRigidity, ChungLiLpUniformRoe, GardellaThielLp, GardellaThielLp2019, KasparovYuPBC, LiaoYu, PhillipsLp, PhillipsSimpLp}. \par
    In this paper, we consider the $\ell^p$ coarse Baum-Connes conjecture for open cones (cf. Definition \ref{DefOpenCone}) and hence obtain a formula to compute the $K$-theory for $\ell^p$ Roe algebras of open cones. The main result of the paper is the following theorem that generalizes N. Higson and J. Roe's result on the coarse Baum-Connes conjecture for open cones in \cite{HigsonRoeCBC} to all $p\in[1,\infty)$.
    \begin{theorem}[cf. Theorem \ref{mainthm}] \label{thmintro}
    Let $\mathcal{O}M$ be the open cone over a compact metric space $M$, then for any $p\in [1,\infty)$, the $\ell^p$ coarse Baum-Connes conjecture holds for $\mathcal{O}M$.
    \end{theorem}
    In \cite{ZZ}, the author and D. Zhou showed that the left-hand side of the $\ell^p$ coarse Baum-Connes conjecture for metric spaces with bounded geometry does not depend on $p\in(1,\infty)$, based on this, we expand this result to all metric spaces and all $p\in[1,\infty)$ in the Section \ref{appendixA}. Thus we have the following corollary.
    \begin{corollary}[cf. Corollary \ref{indepen}] \label{corintro}
    Let $\mathcal{O}M$ be the open cone over a compact metric space $M$, then for any $p\in [1,\infty)$, the group $K_{\ast}(B^p(\mathcal{O}M))$ is isomorphic to the group $K_{\ast}(B^2(\mathcal{O}M))$, i.e. the $K$-theory for the $\ell^p$ Roe algebra of $\mathcal{O}M$ does not depend on $p\in[1,\infty)$. 
    \end{corollary}
    As an application of the theorem and corollary, combined with the result of T. Fukaya and S.-I. Oguni in \cite{FukayaOguniCCH}, we show that the $\ell^p$ coarse Baum-Connes conjecture holds for any proper coarsely convex space (cf. Definition \ref{CoarselyConvex}) and the $K$-theory for $\ell^p$ Roe algebras of such spaces is independent of $p\in[1,\infty)$. The class of proper coarsely convex space includes geodesic Gromov hyperbolic spaces, CAT(0)-spaces, certain Artin groups and Helly groups equipped with the word length metric, see Section \ref{Apps} for more examples.\par
    The paper is organized as follows. In Section \ref{Preliminaries}, we recall some facts of the $\ell^p$ coarse Baum-Connes conjecture for $p\in[1,\infty)$. In Section \ref{MainResults}, we give the proof of Theorem \ref{thmintro} and obtain Corollary \ref{corintro}. In Section \ref{Apps}, we discuss the applications of main results and show some examples. In the end, we discuss the $p$-independency of the left-hand side of the conjecture in the Section \ref{appendixA}.

\subsection*{Acknowledgements}
    The author would like to thank Jiawen Zhang and Dapeng Zhou for providing many valuable suggestions and comments.

\section{Preliminaries} \label{Preliminaries}
    In this section, We briefly recall some facts about $\ell^p$ coarse Baum-Connes conjecture for $1\leq q <\infty$. We refer the reader to \cite{ChungNowakPCBC, ZZ} for more details. \par
    Let $X$ be a proper metric space, i.e. every closed ball in $X$ is compact. The proper metric space is a separable space, since compact metric space is separable. Choose a countable dense subset $Z_X$ in $X$, then there is a natural action $\rho$ of $Bol(X)$ on $\ell^p$-space $\ell^p(Z_X)\otimes \ell^p=\ell^p(Z_X, \ell^p)$ by point-wise multiplication, where $Bol(X)$ is the Banach algebra of all bounded Borel functions on $X$ and $\ell^p$ is the $\ell^p$-space of all $p$-summable sequences on the non-negative integers $\mathbb{N}$. In what follows, We will omit $\rho$ if there is no ambiguity and let $\chi_U$ be the characteristic function on subset $U$ of $X$, and let $p\in [1,\infty)$.

\begin{definition}
    Let $X$ and $Y$ be two proper metric spaces, $Z_X$ and $Z_Y$ be two countable dense subsets of $X$ and $Y$, respectively. Consider a bounded linear operator $T: \ell^p(Z_X)\otimes \ell^p \rightarrow \ell^p(Z_Y)\otimes \ell^p$, define
\begin{enumerate}
\item the \textit{support} of $T$, denoted $\mathrm{supp(T)}$, consists of all points $(x,y)$ in $X\times Y$ such that $\chi_V T \chi_U\neq 0$ for all open neighborhoods $U$ of $x$ and $V$ of $y$.
\item the \textit{propagation} of $T$, denoted $\mathrm{prop(T)}$, is defined to be $\sup\{d(x_1,x_2): (x_1,x_2)\in \mathrm{supp(T)}\}$, here we assume $X$ and $Y$ are same. 
\item $T$ is called to be \textit{locally compact}, if $ T\chi_K$ and $\chi_{K'} T$ are compact operators for all compact subsets $K$ in $X$ and $K'$ in $Y$.
\end{enumerate}
\end{definition}

\begin{definition} \label{DefRoeAlg}
    Let $X$ be a proper metric space, the \textit{$\ell^p$ Roe algebra} of $X$, denoted $B^p(X)$, is defined to be the norm closure of the algebra of all locally compact operators acting on $\ell^p(Z_X)\otimes \ell^p$ with finite propagation.
\end{definition}

\begin{remark}
    The $\ell^p$ Roe algebra $B^p(X)$ is a Banach algebra, when $p=2$, the algebra $B^2(X)$ is a $C^{\ast}$-algebra, called \textit{Roe algebra} (refer to \cite{HigsonRoeBook}\cite{Roe88}\cite{WillettYuBook}). The $\ell^p$ Roe algebra $B^p(X)$ is non-canonically independent of the countable dense subset $Z_X$ of $X$, and its K-theory is canonically independent of $Z_X$ (refer to \cite[Corollary 2.9]{ZZ}).  
\end{remark}

\begin{definition}\label{DefCoarseMap}
    A Borel map $f$ from a proper metric space $X$ to another proper metric space $Y$ is called \textit{coarse} if 
\begin{enumerate}
\item $f$ is proper, i.e. the inverse image of any bounded set is bounded;
\item for any $R\geq 0$, there exists $S\geq 0$ such that $d(f(x),f(x'))\leq S$ for all elements $x,x'\in X$ satisfying $d(x,x')\leq R$ 
\end{enumerate}
\end{definition}

    The following lemma tell us that every coarse map induces a homomorphism between two $\ell^p$ Roe algebras, refer to \cite[Lemma 2.8]{ZZ} for the proof.

\begin{lemma} \label{coarsetohom}
    Let $f$ be a coarse map from a proper metric space $X$ to another proper metric space $Y$, then for any $\epsilon>0$, there exists an isometric operator $V_f: \ell^p(Z_X)\otimes \ell^p \rightarrow \ell^p(Z_Y)\otimes \ell^p$ and a contractive operator $V^+_f: \ell^p(Z_Y)\otimes \ell^p \rightarrow \ell^p(Z_X)\otimes \ell^p$ such that 
                 $$\mathrm{supp}(V_f)\subseteq \{(x,y)\in X\times Y: d(f(x),y)\leq \epsilon\}$$
                 $$\mathrm{supp}(V^+_f)\subseteq \{(y,x)\in Y\times X: d(f(x),y)\leq \epsilon\}.$$
    Moreover, the pair $(V_f,V^+_f)$ gives rise to a homomorphism $ad_f: B^p(X)\rightarrow B^p(Y)$ defined by
                 $$ad_f(T)=V_f T V^+_f$$
for any element $T\in B^p(X)$. \par
    And the map $(ad_f)_{\ast}$ induced by $ad_f$ on K-theory depends only on $f$, not on the choice of the pair $(V_f,V^+_f)$. 
\end{lemma}

\begin{definition} \label{deflocalg}
    Let $X$ be a proper metric space, the \textit{$\ell^p$ localization algebra} of $X$, denoted $B^p_L(X)$, is defined to be the norm closure of the algebra of all bounded and uniformly norm-continuous functions $u$ from $[0,\infty)$ to $B^p(X)$ such that
                $$\mathrm{prop}(u(t))\rightarrow 0, t\rightarrow \infty$$ 
\end{definition}

    Let $f$ be a uniformly continuous coarse map from a proper metric space $X$ to another proper metric space $Y$, $\{\epsilon_k\}$ be a decreasing sequence of positive numbers satisfying $\epsilon_k\rightarrow 0$ as $k\rightarrow \infty$. By the Lemma \ref{coarsetohom}, for each $\epsilon_k$, there exists an isometric operator $V_k$ and a contractive operator $V^+_k$, then for $t\in[0,\infty)$, define
                $$V_f(t)=R(t-k)(V_k\oplus V_{k+1})R^{\ast}(t-k)$$
                $$V^+_f(t)=R(t-k)(V^+_k\oplus V^+_{k+1})R^{\ast}(t-k)$$
for all $k\leq t\leq k+1$, where 
                $$R(t)=\left(
                          \begin{array}{cc}
                              cos(\pi t/2)  & sin(\pi t/2)\\
                              -sin(\pi t/2) & cos(\pi t/2)\\
                          \end{array}
                        \right).      
                $$

    Similar to Lemma \ref{coarsetohom}, we have the following lemma, refer to \cite[Lemma 2.21]{ZZ} for the proof.
\begin{lemma} \label{coarsetohom2}
    Let $f$ and $\{\epsilon_k\}$ be as above, then the pair $(V_f(t),V^+_f(t))$ induces a homomorphism $Ad_f$ from $B^p_L(X)$ to $B^p_L(Y)\otimes M_2(\mathbb{C})$ defined by 
                $$Ad_f(u)(t)=V_f(t)(u(t)\oplus 0)V^+_f(t)$$
for any element $u\in B^p_L(X)$ and $t\in [0,\infty)$, such that
                $$\mathrm{prop}(Ad_f(u)(t))\leq \sup_{(x,x')\in \mathrm{supp}(u(t))}d(f(x),f(x'))+4\epsilon_k$$
for all $t\in [k,k+1]$.
    Moreover, the induced map $(Ad_f)_{\ast}$ on K-theory depends only on $f$ and not on the choice of the pair $(V_f(t),V^+_f(t))$.
\end{lemma}
    
    Now we are ready to formulate the $\ell^p$ coarse Baum-Connes conjecture for $p\in[1,\infty)$. Let $X$ be a proper metric space, consider the evaluation-at-zero homomorphism:
    $$e_0:B_L^p(X)\rightarrow B^p(X)$$
which induces a homomorphism on $K$-theory:
    $$e_0:K_*(B_L^p(X))\rightarrow K_*(B^p(X))$$

    Let $C$ be a locally finite and uniformly bounded cover for  $X$. The \textit{nerve space} $N_C$ associated to $C$ is defined to be the simplicial complex whose set of vertices equals $C$ and where a finite subset $\{U_0,\ldots,U_n\}\subseteq C$ spans an $n$-simplex in $N_C$ if and only if $\bigcap_{i=0}^n U_i \not=\emptyset$. Endow $N_C$ with the \textit{spherical metric}, i.e. the path metric whose restriction to each simplex $\{U_0,\cdots, U_n\}$ is given by
    $$d(\sum^n_{i=0}t_iU_i, \sum^n_{i=0}s_iU_i)=d_{S^n}(\{\frac{t_j}{(\sum_it^2_i)^{1/2}}\}^n_{j=0}, \{\frac{s_j}{(\sum_is^2_i)^{1/2}}\}^n_{j=0}),$$
where $d_{S^n}$ is the standard Riemannian metric on the unit $n$-sphere. The distance of two points which in different connected components is defined to be $\infty$ by convention. \par

\begin{definition}(\cite{RoeBookCoarseCohomology})
    A sequence of locally finite and uniformly bounded covers $\{C_k\}$ of metric space $X$ is called an \textit{anti-\v Cech system} of $X$, if there exists a sequence of positive numbers $R_k\rightarrow \infty$ such that for each $k$, 
\begin{enumerate}
\item every set $U$ in $C_k$ has diameter less than $R_k$;
\item any subset of diameter less than $R_k$ in $X$ is contained in some member of $C_{k+1}$.
\end{enumerate}
\end{definition}

    An anti-\v Cech system always exists (refer to \cite[Lemma 3.15]{RoeBookCoarseCohomology}). By the property of the anti-\v Cech system, for every pair $k_2>k_1$, there exists a simplicial map $i_{k_1 k_2}$ from $N_{C_{k_1}}$ to $N_{C_{k_2}}$ such that $i_{k_1k_2}$ maps a simplex $\{U_0,\ldots,U_n\}$ in $N_{C_{k_1}}$ to a simplex $\{U_0',\ldots,U_n'\}$ in $N_{C_{k_2}}$ satisfying $U_i\subseteq U_i'$ for all $0\leq i\leq n$. Thus by Lemma \ref{coarsetohom},\ref{coarsetohom2}, $i_{k_1 k_2}$ gives rise to the following directed systems of groups:
    \begin{center}
    $(\mathrm{ad}_{i_{k_1k_2}})_*:K_*(B^p(N_{C_{k_1}}))\rightarrow K_*(B^p(N_{C_{k_2}}))$;\\
    $(\mathrm{Ad}_{i_{k_1k_2}})_*:K_*(B_L^p(N_{C_{k_1}}))\rightarrow K_*(B_L^p(N_{C_{k_2}}))$.
    \end{center}

    The following conjecture is called the $\ell^p$ coarse Baum-Connes conjecture.
\begin{conjecture} \label{CBC}
    Let $X$ be a proper metric space, $\{C_k\}_{k=0}^\infty$ be an anti-\v Cech system of $X$, then the evaluation-at-zero homomorphism
$$e_0:\lim_{k\rightarrow \infty} K_*(B_L^p(N_{C_k})) \rightarrow \lim_{k\rightarrow \infty} K_*(B^p(N_{C_k})) \cong K_*(B^p(X))$$
is an isomorphism.
\end{conjecture}

\begin{remark}
    In the above conjecture, every nerve space $N_{C_k}$ is coarse equivalent to $X$, i.e. there exist two coarse maps $f:N_{C_k}\rightarrow X$ and $g: X\rightarrow N_{C_k}$ such that $d(gf(y),id_{N_{C_k}}(y))\leq R$ and $d(fg(x),id_{X}(x))\leq R$ for some constant $R$, then by Lemma \ref{coarsetohom}, it is not difficult to prove that $\lim_{k\rightarrow \infty} K_*(B^p(N_{C_k})) \cong K_*(B^p(X))$. Moreover, the $\ell^p$ coarse Baum-Connes conjecture for $X$ does not depend on the choice of the anti-\v Cech system. When $p=2$, the above conjecture is the known coarse Baum-Connes conjecture and has important applications to topology and geometry (refer to \cite{HigsonRoeBook}\cite{WillettYuBook}). 
\end{remark}

    Let $B_{L,0}^p(X)=\{u\in B_L^p(X):u(0)=0\}$. There exists an exact sequence:
    $$0 \rightarrow B_{L,0}^p(X) \rightarrow B_L^p(X) \rightarrow B^p(X) \rightarrow 0$$
    Thus by six-term exact sequence in K-theory, we have the following reduction:
\begin{lemma}\label{Thm:VanishingObstruction}
    Let $X$ be a proper metric space, $\{C_k\}_{k=0}^\infty$ be an anti-\v Cech system of $X$, then the $\ell^p$ coarse Baum-Connes conjecture is true if and only if
         $$\lim_{k\rightarrow \infty} K_*(B_{L,0}^p(N_{C_k}))=0$$
\end{lemma}

    The following lemma tells us that the left-hand side of the $\ell^p$ coarse Baum-Connes conjecture is independent of $p$, we will prove it in the Section \ref{appendixA}.
\begin{lemma} \label{indepofloc}
    Let $X$ be a proper metric space, then for any $p\in [1,\infty)$, the group $K_{\ast}(B^p_L(X))$ is isomorphic to the group $K_{\ast}(B^2_L(X))$, i.e. the group $K_{\ast}(B^p_L(X))$ does not depend on $p\in [1,\infty)$.
\end{lemma}
 
   Combine the above lemma with the $\ell^p$ coarse Baum-Connes conjecture, we have the following result concerning the $p$-independency of $K$-theory for $\ell^p$ Roe algebras.
    
\begin{corollary} \label{indepRoe}
    Let $X$ be a proper metric space. If for all $p\in [1,\infty)$, the $\ell^p$ coarse Baum-Connes conjecture is true for $X$, then the K-theory of $\ell^p$ Roe algebra $K_{\ast}(B^p(X))$ does not depend on $p\in [1,\infty)$.
\end{corollary}

\section{Main Results} \label{MainResults}
    In this section, we will prove that the $\ell^p$ coarse Baum-Connes conjecture holds for open cones, thus the K-theory of $\ell^p$ Roe algebras of open cones is independent of $p\in [1,\infty)$. The idea of the proof comes from \cite{HigsonRoeCBC}. Firstly, we recall a concept in coarse geometry.
    
\begin{definition} \label{DefCoarHom}
    Let $f,g: X \rightarrow Y$ be two coarse maps between proper metric spaces, $f$ and $g$ are called to be \textit{coarsely homotopic}, if there exists a metric subspace $Z=\{(x,t): 0\leq t\leq t_x\}$ of $X\times \mathbb{R}$ and a coarse map $h: Z\rightarrow Y$, such that 
    \begin{enumerate}
        \item the map from $X$ to $\mathbb{R}$ given by $x\mapsto t_x$ satisfies that for any $R\geq 0$, there exists $S\geq 0$ such that $|t_{x}-t_{x'}|\leq S$ for all elements $x,x'\in X$ with $d(x,x')\leq R$;
        \item $h(x,0)=f(x)$;
        \item $h(x,t_x)=g(x)$.
    \end{enumerate}
    The map $f$ is called a \textit{coarse homotopy equivalence map} if there exists a coarse map $f':Y\rightarrow X$ such that $f'f$ and $ff'$ are coarsely homotopic to the identities $id_X$ and $id_Y$, respectively. Call $X$ and $Y$ are \textit{coarsely homotopy equivalent} if there exists a coarse homotopy equivalence map from $X$ to $Y$.
\end{definition}

    The following lemma implies that the $\ell^p$ coarse Baum-Connes conjecture is permanent under the coarse homotopy equivalence. The proof is relied on Mayer-Vietoris principle, refer to \cite[Proposition 12.4.12]{HigsonRoeBook} for more details.
\begin{lemma}\label{CoaHomCBC}
    For any $p\in[1,\infty)$, let $X$, $Y$ be two proper metric spaces and $f,g: X\rightarrow Y$ be two coarse maps, let $KX^p_{\ast}(X)$ and $KX^p_{\ast}(Y)$ represent the left side of the $\ell^p$ coarse Baum-Connes conjecture (i.e. Conjecture \ref{CBC}) for $X$ and $Y$, respectively. If $f$ is coarsely homotopic to $g$, then they induce same homomorphisms: $f_{\ast}=g_{\ast}: KX^p_{\ast}(X)\rightarrow KX^p_{\ast}(Y)$ and $(ad_f)_{\ast}=(ad_g)_{\ast}:K_{\ast}(B^p(X))\rightarrow K_{\ast}(B^p(Y))$. Moreover, if $f$ is a coarse homotopy equivalence map, then we have the following commutative diagram and two vertical homomorphisms are isomorphisms
    \begin{displaymath}
      \xymatrix{
                KX^p_{\ast}(X)\ar[r]^{e_0\:\:\:\:}\ar[d]_{f_{\ast}}^{\cong} & K_{\ast}(B^p(X))\ar[d]_{(ad_f)_{\ast}}^{\cong} \\
                KX^p_{\ast}(Y)\ar[r]^{e_0\:\:\:\:}                                   & K_{\ast}(B^p(Y))
               }
    \end{displaymath}
\end{lemma}
                

    Secondly, let us recall the definition of open cones.
\begin{definition} \label{DefOpenCone}
    Let $(M,d_M)$ be a compact metric space with diameter at most 2, the \textit{open cone} over $M$, denoted $\mathcal{O}M$, is defined to be the quotient space $\mathbb{R}_{\geq 0}\times M/(\{0\}\times M)$ with the following metric
                    $$d((t,x),(s,y))=|t-s|+\min{\{t,s\}}d_M(x,y)$$
for any $(t,x),(s,y)\in \mathcal{O}M$. 
\end{definition}

    Obviously, the open cone is a proper metric space. By \cite[Proposition B.1]{FukayaOguniCBCRH} and \cite[Proposition 4.3]{HigsonRoeCBC}, We can simplify the left side of the $\ell^p$ coarse Baum-Connes conjecture for open cones by the following lemma.
\begin{lemma} \label{SimplifyCBC}
    Let $\mathcal{O}M$ be the open cone over a compact metric space $M$, then there exists an anti-\v{C}ech system $\{C_k\}$ of $\mathcal{O}M$ such that 
                    $$i_{\ast}: K_{\ast}(B^p_L(\mathcal{O}M))\rightarrow \lim_{k\rightarrow \infty}K_{\ast}(B^p_L(N_{C_k}))$$
is an isomorphism, where $i_{\ast}$ is induced by a family of maps $i_k:\mathcal{O}M \rightarrow N_{C_k}$ that maps an element $x$ in $\mathcal{O}M$ to an element $U^{(i)}_k$ in $N_{C_k}$ such that $x\in U^{(i)}_k$. \par
    Moreover, the following diagram is commutative
                 \begin{displaymath}
                    \xymatrix{
                         K_{\ast}(B^p_L(\mathcal{O}M)) \ar[d]_{i_{\ast}} \ar[dr]^{e_0}  &&\\
                         \lim_{k\rightarrow \infty} K_{\ast}(B^p_L(N_{C_k})) \ar[r]_{\:\:\:\:\:\:\:\:\:\:e_0} &K_{\ast}(B^p(\mathcal{O}M)).
                      }
                  \end{displaymath}
    Thus, the map $i_{\ast}$ induces an isomorphism from the group $K_{\ast}(B^p_{L,0}(\mathcal{O}M))$ to the group $\lim_{k\rightarrow \infty} K_{\ast}(B^p_{L,0}(N_{C_k})).$
\end{lemma} 

    Combining the above lemma with Lemma \ref{CoaHomCBC}, we have the following lemma.
\begin{lemma}\label{SameHom}
    Let $\mathcal{O}M$ be the open cone over a compact metric space $M$ and $f: \mathcal{O}M\rightarrow \mathcal{O}M$ be a coarse map defined by $f((t,x))=(\frac{t}{2},x)$ for any $(t,x)\in \mathcal{O}M$, then $f$ is coarsely homotopic to the identity map $i$ and they induce the same homomorphism on $K_{\ast}(B^p_{L,0}(\mathcal{O}M))$, where $B^p_{L,0}(\mathcal{O}M)=\{u\in B^p_L(\mathcal{O}M): u(0)=0\}$.
\end{lemma}
\begin{proof}
    Let $Z=\{((t,x),s):0\leq s\leq \frac{1}{2}\}$ and $h: Z\rightarrow \mathcal{O}M$ defined by $h((t,x),s)=((1-s)t,x)$, then $h$ is a coarse homotopy connecting $i$ and $f$. Let $\{C_k\}$ be an anti-\v{C}ech system of $\mathcal{O}M$, then by Lemma \ref{CoaHomCBC}, $f$ and $i$ induce same homomorphisms on $\lim_{k\rightarrow \infty} K_{\ast}(B^p_L(N_{C_k}))$ and $K_{\ast}(B^p(\mathcal{O}M))$, thus by the five lemma, they induce the same homomorphism on $\lim_{k\rightarrow \infty} K_{\ast}(B^p_{L,0}(N_{C_k}))$, then by Lemma \ref{SimplifyCBC} and the five lemma, they induce the same homomorphism on $K_{\ast}(B^p_{L,0}(\mathcal{O}M))$.
\end{proof}

    Now, we begin to show and prove the main theorem of this article.
\begin{theorem} \label{mainthm}
    Let $\mathcal{O}M$ be the open cone over a compact metric space $M$, then for any $p\in [1,\infty)$, the $\ell^p$ coarse Baum-Connes conjecture holds for $\mathcal{O}M$.
\end{theorem}
\begin{proof}
    By Lemma \ref{Thm:VanishingObstruction} and Lemma \ref{SimplifyCBC}, it is sufficient to prove $K_{\ast}(B^p_{L,0}(\mathcal{O}M))=0$.\par
    Choose a countable dense subset $Z_M$ in $M$, let $E^p=\ell^p(\mathbb{Q}_{\geq 0}\times Z_M)\otimes \ell^p$ and $E^{p,\infty}=\bigoplus^{\infty}_{n=0}(\ell^p(\mathbb{Q}_{\geq 0}\times Z_M)\otimes \ell^p)$, where $\bigoplus$ is the $\ell^p$-direct sum, then there exists a natural action of $Bol(\mathcal{O}M)$ on $E^{p,\infty}$, thus similar to the definition of the $\ell^p$ Roe algebra on $E^p$ (denoted $B^p(\mathcal{O}M;E^p)$), we can define the $\ell^p$ Roe algebra on $E^{p,\infty}$, denoted $B^p(\mathcal{O}M;E^{p,\infty})$, so do $B^p_L(\mathcal{O}M;E^{p,\infty})$ and $B^p_{L,0}(\mathcal{O}M;E^{p,\infty})$. Denote $B^p_{L,0}(\mathcal{O}M)$ by $B^p_{L,0}(\mathcal{O}M;E^p)$. \par
    Let $f$ be a map from $\mathcal{O}M$ to $\mathcal{O}M$ given by $f((t,x))=(\frac{t}{2},x)$, define two linear operators $V$,$V^+$ on $E^p$ by 
                  $V(\delta_{(t,x)})=\delta_{(\frac{t}{2},x)}$,
                  $V^+(\delta_{(t,x)})=\delta_{(2t,x)}$,
where $\delta_{(t,x)}\in E^p$ maps $(t',x')$ to $1$ for $t'=t$ and $x'=x$, to $0$ for others. Then we have $V^+ V=V V^+=I$ and $\mathrm{supp}(V)=\{((t,x),f((t,x))):(t,x)\in \mathcal{O}M\}, \mathrm{supp}(V^+)=\{(f((t,x)),(t,x)): (t,x)\in \mathcal{O}M\}$. \par
    Define a map $\phi: B^p_{L,0}(\mathcal{O}M;E^p)\rightarrow B^p_{L,0}(\mathcal{O}M;E^{p,\infty})$ by 
                  $$\phi(u)(s)=0\oplus V u(s) V^+\oplus \cdots \oplus V^{n+1} u(s-n) (V^+)^{n+1}\oplus \cdots$$
for any $u\in B^p_{L,0}(\mathcal{O}M;E^p)$ and let $u(s)=0$ for $s<0$. Now we show $\phi$ is well-defined. Firstly, for any $s\geq 0$, the operator $\phi(u)(s)$ is locally compact since $u(s-n)=0$ for all $n\geq s$. Secondly, by direct computation, we have that $\mathrm{supp}(V^{n+1} u(s-n) (V^+)^{n+1})$ is equal to the set 
                  $$\{((\frac{t}{2^{n+1}},x),(\frac{t'}{2^{n+1}},x')): ((t,x),(t',x'))\in \mathrm{supp}(u(s-n))\},$$
which implies that 
                  $$\mathrm{prop}(V^{n+1} u(s-n) (V^+)^{n+1})=\frac{\mathrm{prop}(u(s-n))}{2^{n+1}},$$
thus $\mathrm{prop}(\phi(u)(s))\rightarrow 0$ as $s\rightarrow \infty$. Finally, $\phi(u)(0)=0$ is obvious, thus $\phi$ is well-defined. \par
    Define $\psi:B^p_{L,0}(\mathcal{O}M;E^p)\rightarrow B^p_{L,0}(\mathcal{O}M;E^{p,\infty})$ by
                  $$\psi(u)(s)=u(s)\oplus 0 \oplus \cdots \oplus 0 \oplus \cdots$$
for any $u\in B^p_{L,0}(\mathcal{O}M;E^p)$. Now we prove $\psi$ induces an isomorphism $\psi_{\ast}$ on K-theory level. Let $T: E^p\rightarrow E^{p,\infty}$ given by $T(\xi)=(\xi,0,\cdots,0,\cdots)$ for any $\xi \in E^p$ and $T^+:E^{p,\infty}\rightarrow E^p$ given by $T^+(\xi_0,\xi_1,\cdots,\xi_n,\cdots)=\xi_0$ for any $(\xi_0,\xi_1,\cdots,\xi_n,\cdots)\in E^{p,\infty}$, then $T^+T=I$ and $\psi(u)(s)=Tu(s)T^+$. Identify $E^{p,\infty}$ with $\ell^p(\mathbb{Q}_{\geq 0}\times Z_M)\otimes (\bigoplus \ell^p)$, then the natural isometric isomorphism between $\ell^p$ and $\bigoplus \ell^p$ gives an isometric isomorphism $U$ from $E^p$ to $E^{p,\infty}$ which is the identity operator on $\ell^p(\mathbb{Q}_{\geq 0}\times Z_M)$. Define
                  $$W=\left(
                          \begin{array}{cc}
                              TU^{-1}  & I-TT^+\\
                              0        & UT^+\\
                          \end{array}
                        \right),     
                  $$
then $W$ is an invertible operator on $E^{p,\infty}\oplus E^{p,\infty}$ with zero propagation and we have
                  $$\psi(u)(s)\oplus 0=W(U u(s)U^{-1}\oplus 0)W^{-1}$$
for any $u\in B^p_{L,0}(\mathcal{O}M;E^p)$, thus $\psi_{\ast}$ is an isomorphism from $K_{\ast}(B^p_{L,0}(\mathcal{O}M;E^p))$ to $K_{\ast}(B^p_{L,0}(\mathcal{O}M;E^{p,\infty}))$. \par 
    The homomorphism $\phi+\psi: B^p_{L,0}(\mathcal{O}M;E^p)\rightarrow B^p_{L,0}(\mathcal{O}M;E^{p,\infty})$ is given by
                 $$(\phi+\psi)(u)(s)=u(s)\oplus V u(s) V^+\oplus \cdots \oplus V^{n+1} u(s-n) (V^+)^{n+1}\oplus \cdots.$$
Define a homotopy $\Psi_{\lambda}:B^p_{L,0}(\mathcal{O}M;E^p)\rightarrow B^p_{L,0}(\mathcal{O}M;E^{p,\infty})$ by 
                 $$(\Psi)_{\lambda}(u)(s)=u(s)\oplus V u(s-\lambda) V^+\oplus \cdots \oplus V^{n+1} u(s-n-\lambda) (V^+)^{n+1}\oplus \cdots,$$
where $\lambda\in[0,1]$, then $\Psi_0=\phi+\psi$ and $S(\bigoplus_{n=0}^{\infty}V)(\Psi_1(u)(s))(\bigoplus_{n=0}^{\infty}V^+)S^+=\phi(u)(s)$ for any $u\in B^p_{L,0}(\mathcal{O}M;E^p)$, where $S,S^+:E^{p,\infty}\rightarrow E^{p,\infty}$ are defined by $S(\xi_0,\xi_1,\cdots)=(0,\xi_0,\xi_1,\cdots)$ and $S^+(\xi_0,\xi_1,\xi_2,\cdots)=(\xi_1,\xi_2,\cdots)$. Thus we have the following relations at $K$-theory level
                  $$\phi_{\ast}+\psi_{\ast}=(\Psi_0)_{\ast}=(\Psi_1)_{\ast}=\phi_{\ast}$$
(the last equation as above is due to Lemma \ref{coarsetohom2} and \ref{SameHom}), this implies that $\psi_{\ast}=0$. But we have shown that $\psi_{\ast}$ is an isomorphism from $K_{\ast}(B^p_{L,0}(\mathcal{O}M;E^p))$ to $K_{\ast}(B^p_{L,0}(\mathcal{O}M;E^{p,\infty}))$, Therefore, $K_{\ast}(B^p_{L,0}(\mathcal{O}M))=0$.                
\end{proof}
\begin{remark}
    A proper metric space $X$ is \textit{scaleable} if there is a continuous and proper map $f: X\rightarrow X$ which is coarsely homotopic to the identity map, such that $d(f(x),f(x'))\leq \frac{1}{2}d(x,x')$ for all $x,x'\in X$. By Lemma \ref{SameHom}, every open cone is scaleable space. Actually, the above theorem widely holds for all scaleable spaces by the similar proof. In \cite{HigsonRoeCBC}, N. Higson and J. Roe have proved this theorem for $p=2$ by using different language.
\end{remark}
    Combine the above theorem and Corollary \ref{indepRoe}, we have the following corollary.
\begin{corollary} \label{indepen}
    Let $\mathcal{O}M$ be the open cone over a compact metric space $M$, then for any $p\in [1,\infty)$, the group $K_{\ast}(B^p(\mathcal{O}M))$ is isomorphic to the group $K_{\ast}(B^2(\mathcal{O}M))$, i.e. the $K$-theory for $\ell^p$ Roe algebra of $\mathcal{O}M$ does not depend on $p\in[1,\infty)$. 
\end{corollary}

\section{Applications} \label{Apps}
    In this section, we will show some applications of the main theorem (i.e. Theorem \ref{mainthm}) based on the result of T. Fukeya and S.-I. Oguni in \cite{FukayaOguniCCH}. 

    The following concept comes from \cite[Definition 3.1]{FukayaOguniCCH}.
\begin{definition} \label{CoarselyConvex}
    Let $X$ be a metric space. Let $\lambda\geq 1$, $k\geq 0$, $E\geq 1$ and $C\geq 0$ be constants. Let $\theta: \mathbb{R}_{\geq 0}\rightarrow \mathbb{R}_{\geq 0}$ be a non-decreasing function. Let $\mathcal{L}$ be a family of $(\lambda,k)$-quasi-geodesic segments. The metric space $X$ is \textit{$(\lambda,k,E,C,\theta,\mathcal{L})$-coarsely convex}, if $\mathcal{L}$ satisfies the following:
    \begin{enumerate}
        \item for $x_1,x_2\in X$, there exists a quasi-geodesic segment $\gamma\in \mathcal{L}$ with $\gamma:[0,a]\rightarrow X$ such that $\gamma(0)=x_1$ and $\gamma(a)=x_2$;
        \item let $\gamma, \eta\in \mathcal{L}$ be quasi-geodesic segments with $\gamma:[0,a]\rightarrow X$ and $\eta:[0,b]\rightarrow X$, then for $t\in [0,a]$, $s\in [0,b]$ and $0\leq c\leq1$, we have that 
             $$d(\gamma(ct),\eta(cs))\leq cEd(\gamma(t),\eta(s))+(1-c)Ed(\gamma(0),\eta(0))+C;$$
        and 
             $$|t-s|\leq \theta(d(\gamma(0),\eta(0))+d(\gamma(t),\eta(s))).$$
    \end{enumerate} 
    We call a metric space $X$ to be a \textit{coarsely convex space}, if there exist data $\lambda,k,E,C,\theta,\mathcal{L}$ such that $X$ is $(\lambda,k,E,C,\theta,\mathcal{L})$-coarsely convex.
\end{definition}

\begin{remark}
    For a $(\lambda,k,E,C,\theta,\mathcal{L})$-coarsely convex space $X$, in \cite[Section 4]{FukayaOguniCCH}, T. Fukeya and S.-I. Oguni constructed the ideal boundary $\partial X$ of $X$, as the set of equivalence classes of quasi-geodesic rays which can be approximated by quasi-geodesic segments in $\mathcal{L}$. 
\end{remark}
    
    The following result is the main theorem in \cite{FukayaOguniCCH}.
\begin{lemma}
    Let $X$ be a proper coarsely convex space, then $X$ is coarsely homotopy equivalent to $\mathcal{O}\partial X$, the open cone over the ideal boundary of $X$.
\end{lemma}

    Combine this lemma with Theorem \ref{mainthm} and Lemma \ref{CoaHomCBC}, we have the following result.
\begin{theorem}\label{CoaConvCBC}
    Let $X$ be a proper coarsely convex space, then $X$ satisfies the $\ell^p$ coarse Baum-Connes conjecture for any $p\in [1,\infty)$.
\end{theorem}
\begin{remark}
    In \cite{FukayaOguniCCH}, T. Fukaya and S.-I. Oguni have proved this theorem for $p$=2 by using N. Higson and J. Roe's result in \cite{HigsonRoeCBC}.
\end{remark}

    By the Corollary \ref{indepRoe}, we have the following corollary.
\begin{corollary}
    Let $X$ be a proper coarsely convex space, then the $K$-theory of $\ell^p$ Roe algebra $K_{\ast}(B^p(X))$ does not depend on $p\in [1,\infty)$.
\end{corollary}

\begin{example}
    The following examples are proper coarsely convex spaces:
    \begin{enumerate}
      \item geodesic Gromov hyperbolic spaces \cite{GhysHarpeHyperbolic};
      \item CAT(0)-spaces, more generally, Busemann non-positively curved spaces \cite{BridsonHaefligerBook}\cite{PapadopoulosBook};
      \item systolic groups with the word length metric \cite{OsajdaPrzytyckiSystolic}, especially, Artin groups of almost large type \cite{HuangOsajdaArtin} and graphical small cancellation groups \cite{OsajdaPrytulaSys}; 
      \item Helly groups with the word length metric \cite{HellyGroups}, especially, weak Garside groups of finite type and FC-type Artin groups \cite{HuangOsajdaHelly};
      \item products of proper coarsely convex spaces.
    \end{enumerate}
    By the above theorem and corollary, these spaces satisfy the $\ell^p$ coarse Baum-Connes conjecture for any $p\in[1,\infty)$ and the $K$-theory of their $\ell^p$ Roe algebras does not depend on $p\in[1,\infty)$.
\end{example}


\section{Proof of Lemma \ref{indepofloc}} \label{appendixA}
    In this section, we will give a proof of Lemma \ref{indepofloc} based on the results of Section 5 in \cite{ZZ} where the authors have proved this lemma for all finite dimensional simplicial complexes. This answers a question raised by Y. C. Chung and P. W. Nowak in \cite{ChungNowakPCBC}.\par
    In what follows, we assume $X$ be a locally compact, second countable, Hausdorff space, $X^+$ be the one-point compactification of $X$ and $Z_X$ be a countable dense subset in $X$. Firstly, Similar to \cite[Definition 6.2.3]{WillettYuBook}, we extend the definition of the $\ell^p$ localization algebra to all locally compact, second countable, Hausdorff spaces.
    
\begin{definition} \label{deflocalg2}
    The algebra of all bounded linear operators on $\ell^p(Z_X)\otimes \ell^p$ is denoted by $\mathfrak{B}(\ell^p(Z_X)\otimes \ell^p)$. Define $B^p_L[X]$ to be the collection of all bounded functions $u$ from $[0,\infty)$ to $\mathfrak{B}(\ell^p(Z_X)\otimes \ell^p)$ such that
  \begin{enumerate}
    \item for any compact subset $K$ of $X$, there exists a positive number $t_K$ such that $\chi_K u(t)$ and $u(t)\chi_K$ are compact operators for all $t\geq t_K$, moreover, the functions $t\mapsto \chi_K u(t)$ and $t\mapsto u(t)\chi_K$ are uniformly norm continuous when restricted to $[t_K,\infty)$,
    \item for any open neighborhood $U$ of the diagonal in $X^+\times X^+$, there exists a positive number $t_U$ such that $\mathrm{supp}(u(t))\subseteq U$ for all $t\geq t_U$.
  \end{enumerate}
  The \textit{$\ell^p$ localization algebra} of $X$, denoted $B^p_L(X)$, is defined to be the norm closure of $B^p_L[X]$ with the norm $||u||=\sup||u(t)||$.  
\end{definition}

\begin{remark} \label{isohomology}
    The above definition is non-canonically independent of the countable dense subset $Z_X$ of $X$, and its K-theory is canonically independent of $Z_X$. When $X$ is a proper metric space, the two $\ell^p$ localization algebras of $X$ defined by Definition \ref{deflocalg} and Definition \ref{deflocalg2} are isomorphic at the $K$-theory level, refer to Chapter 6 of \cite{WillettYuBook} for the details. 
\end{remark}

    For $p\in (1,\infty)$, let $q$ be the dual number of $p$, i.e. $\frac{1}{p}+\frac{1}{q}=1$, define $\mathfrak{B}^{\ast}(\ell^p(Z_X)\otimes \ell^p)$ to be the collection of all linear operators $T$ acting on $C_c(Z_X\times \mathbb{N})$ and satisfying that there exists a constant $C$ such that $||T\xi||_{\ell^p(Z_X)\otimes \ell^p}\leq C||\xi||_{\ell^p(Z_X)\otimes \ell^p}$ and $||T\xi||_{\ell^q(Z_X)\otimes \ell^q}\leq C||\xi||_{\ell^q(Z_X)\otimes \ell^q}$ for any $\xi \in C_c(Z_X\times \mathbb{N})$, where $C_c(Z_X\times \mathbb{N})$ is the linear space of all continuous functions on $Z_X\times \mathbb{N}$ with compact support. Then $\mathfrak{B}^{\ast}(\ell^p(Z_X)\otimes \ell^p)$ is a Banach algebra equipped with the norm $||T||_{\mathrm{max}}=\max\{||T||_{\mathfrak{B}(\ell^p(Z_X)\otimes \ell^p)},||T||_{\mathfrak{B}(\ell^q(Z_X)\otimes \ell^q)}\}$. For $p=1$, let $q=3$, then we can similarly define $\mathfrak{B}^{\ast}(\ell^1(Z_X)\otimes \ell^1)$ to be the collection of all linear operators $S$ acting on $C_c(Z_X\times \mathbb{N})$ which can be boundedly extend to $\ell^1(Z_X)\otimes \ell^1$ and $\ell^3(Z_X)\otimes \ell^3$.

\begin{definition} \label{defduallocalg}
    Let $p,q$ be as above. Define $B^{p,\ast}_L[X]$ to be the collection of all bounded functions $u$ from $[0,\infty)$ to $\mathfrak{B}^{\ast}(\ell^p(Z_X)\otimes \ell^p)$ such that
  \begin{enumerate}
    \item for any compact subset $K$ of $X$, there exists a positive number $t_K$ such that $\chi_K u(t)$ and $u(t)\chi_K$ are compact operators on $\ell^p(Z_X)\otimes \ell^p$ and $\ell^q(Z_X)\otimes \ell^q$ for all $t\geq t_K$, moreover, the functions $t\mapsto \chi_K u(t)$ and $t\mapsto u(t)\chi_K$ are uniformly norm continuous when restricted to $[t_K,\infty)$,
    \item for any open neighborhood $U$ of the diagonal in $X^+\times X^+$, there exists a positive number $t_U$ such that $\mathrm{supp}(u(t))\subseteq U$ for all $t\geq t_U$.
  \end{enumerate}
  The \textit{dual $\ell^p$ localization algebra} of $X$, denoted $B^{p,\ast}_L(X)$, is defined to be the norm closure of $B^p_L[X]$ with the norm $||u||=\sup||u(t)||_{\mathrm{max}}$.  
\end{definition}

    There is an inclusion homomorphism $i: B^{p,\ast}_L(X)\rightarrow B^p_L(X)$, on the other hand, by the Riesz-Thorin interpolation theorem we also have a contractive homomorphism $\varphi: B^{p,\ast}_L(X)\rightarrow B^2_L(X)$. In \cite{ZZ}, the authors have proved that $i$ and $\varphi$ induce two group isomorphisms at the $K$-theory level for any finite dimensional simplicial complex $X$, more general we have the following lemma.
    
\begin{lemma}
    Let $i$ and $\varphi$ be two homomorphisms as above, then they induce two group isomorphisms at the $K$-theory level for all locally compact, second countable, Hausdorff space $X$.
\end{lemma}
\begin{proof}
    The proof for $i$ and the proof for $\varphi$ are very similar, so we just prove the lemma for $\varphi$. By \cite[Proposition 5.18]{ZZ}, we have the following claim.\par 
    \textbf{Claim}: the homomorphism $\varphi$ induces a group isomorphism at the $K$-theory level for any finite dimensional simplicial complex. \par
    Let $X^+$ be the one-point compactification of $X$, then we have $K_{\ast}(B^{p,\ast}_L(X))\cong K_{\ast}(B^{p,\ast}_L(X^+))/K_{\ast}(B^{p,\ast}_L(\{\infty\}))$ and $K_{\ast}(B^2_L(X))\cong K_{\ast}(B^2_L(X^+))/K_{\ast}(B^2_L(\{\infty\}))$, then by the Claim, we just need to prove the lemma for compact spaces. Thus, in what follows, we assume $X$ is a compact space. \par
    We can endow $X$ with a metric that induces the topology since $X$ is a metrizable space, then $X$ can be represented as the inverse limit of a sequence of finite simplicial complexes by the following construction. Fix a sequence of finite covers $\{C_k\}$ such that $\sup_{U\in C_k}\mathrm{diam (U)}\rightarrow 0$ as $k\rightarrow \infty$ and such that each element of $C_{k+1}$ is contained in an element of $C_k$, then a sequence of nerve spaces of $C_k$ produces an inverse system $\cdots\rightarrow N_{C_3}\rightarrow N_{C_2}\rightarrow N_{C_1}$ and $X$ is its inverse limit.  \par
    By the Chapter 6 of \cite{WillettYuBook} and Section 5 of \cite{ZZ}, the two families of functors from the category of locally compact, second countable, Hausdorff spaces to the category of abelian groups that are given by $F^{p,\ast}_{\ast}: X\mapsto K_{\ast}(B^{p,\ast}_L(X))$ and $F^2_{\ast}: X\mapsto K_{\ast}(B^2_L(X))$, respectively, satisfy the axioms of generalized Steenrod homology theory, i.e. (1) $F^{p,\ast}_{\ast}$ and $F^2_{\ast}$ are homotopy functors; (2) $F^{p,\ast}_{\ast}(X)$ and $F^2_{\ast}(X)$ have Mayer-Vietoris sequence for a pair of closed subsets of $X$; (3) cluster axiom: if $X$ is a countable disjoint union of spaces $X_j$, then the inclusions $X_j\rightarrow X$ induce two families of isomorphisms $\prod_j F^{p,\ast}_{\ast}(X_j)\cong F^{p,\ast}_{\ast}(X)$ and $\prod_j F^2_{\ast}(X_j)\cong F^2_{\ast}(X)$. Then By \cite[Proposition 7.3.4]{HigsonRoeBook}, we have the following commutative diagram about the Milnor exact sequences:  
    \begin{displaymath}
           \xymatrix{
0\ar[r] & {\lim\limits_{\leftarrow}}^1 K_{\ast+1}(B^{p,\ast}_L(N_{C_k}))\ar[r] \ar[d]_{{\lim\limits_{\leftarrow}}^1 \varphi_{\ast}}  & K_{\ast}(B^{p,\ast}_L(X)) \ar[r] \ar[d]_{\varphi_{\ast}}  & \lim\limits_{\leftarrow}K_{\ast}(B^{p,\ast}_L(N_{C_k}))\ar[r] \ar[d]_{\lim\limits_{\leftarrow} \varphi_{\ast}} & 0  \\
0\ar[r] & {\lim\limits_{\leftarrow}}^1 K_{\ast+1}(B^2_L(N_{C_k}))\ar[r] & K_{\ast}(B^2_L(X))\ar[r] &  \lim\limits_{\leftarrow}K_{\ast}(B^2_L(N_{C_k}))\ar[r] & 0
                    }
    \end{displaymath}
Thus by the five lemma and the Claim, we complete the proof.
\end{proof}

    Finally, the Lemma \ref{indepofloc} is an obvious corollary of the above lemma.

\bibliographystyle{plain}
\bibliography{pCBCOC}

\end{document}